\documentclass[11pt,reqno]{amsart}
\usepackage{fullpage}
\usepackage{mathrsfs,amssymb,graphicx,verbatim,amsmath,amsfonts}
\usepackage{paralist}
\usepackage{esint}
\usepackage[breaklinks,pdfstartview=FitH]{hyperref}
\usepackage{upgreek}
\usepackage[mathscr]{euscript}
\usepackage{dutchcal}
\usepackage[T1]{fontenc}

\usepackage{graphicx}
\usepackage{sidecap}
\usepackage[all]{xy}
\usepackage[font=small,labelfont=bf,labelsep=endash]{caption}
\usepackage{enumitem}

\renewcommand{\eta}{\upeta}

\renewcommand{\epsilon}{\upepsilon}
\newcommand{\MM}{\mathcal{M}}

\renewcommand{\chi}{\upchi}

\renewcommand{\le}{\leqslant}
\renewcommand{\ge}{\geqslant}

\renewcommand{\setminus}{\smallsetminus}
\renewcommand{\gamma}{\upgamma}
\renewcommand{\lambda}{\uplambda}
\renewcommand{\alpha}{\upalpha}
\renewcommand{\delta}{\updelta}
\renewcommand{\beta}{\upbeta}
\renewcommand{\omega}{\upomega}
\renewcommand{\nu}{\upnu}
\renewcommand{\mu}{\upmu}
\renewcommand{\psi}{\uppsi}
\renewcommand{\phi}{\upphi}
\renewcommand{\rho}{\uprho}
\renewcommand{\kappa}{\upkappa}

\renewcommand{\tau}{\uptau}

\makeatletter
\def\moverlay{\mathpalette\mov@rlay}
\def\mov@rlay#1#2{\leavevmode\vtop{%
   \baselineskip\z@skip \lineskiplimit-\maxdimen
   \ialign{\hfil$\m@th#1##$\hfil\cr#2\crcr}}}
\newcommand{\charfusion}[3][\mathord]{
    #1{\ifx#1\mathop\vphantom{#2}\fi
        \mathpalette\mov@rlay{#2\cr#3}
      }
    \ifx#1\mathop\expandafter\displaylimits\fi}
\makeatother

\newcommand{\ee}{\mathsf{e}}

\newcommand{\NN}{\mathcal{N}}

\newcommand{\n}{\{1,\ldots,n\}}

\newcommand{\sub}{\mathscr{C}}

\renewcommand{\d}{\delta}

\newcommand{\R}{\mathbb R}

\newtheorem{theorem}{Theorem}

\newtheorem{claim}[theorem]{Claim}

\theoremstyle{remark}
\newtheorem{remark}[theorem]{Remark}

\renewcommand{\pi}{\uppi}
\renewcommand{\zeta}{\upzeta}

\renewcommand{\subset}{\subseteq}

\renewcommand{\sigma}{\upsigma}

\newcommand{\N}{\mathbb N}

\newcommand{\eqdef}{\stackrel{\mathrm{def}}{=}}

\newcommand{\Lip}{\mathrm{Lip}}

\begin{document}

\title{A relation between finitary Lipschitz extension moduli}
\thanks{M.~M. was supported by the BSF. A.~N. was supported by the BSF, the Packard Foundation and the Simons Foundation. The research that is presented here was conducted under the auspices of the Simons Algorithms and Geometry (A\&G) Think Tank.}

\author{Manor Mendel}
\address{Mathematics \& Computer Science Department, Open University, 1 University Road, P.O. Box 808, Raanana 53537, ISRAEL.}
\email{mendelma@gmail.com}
\author{Assaf Naor}
\address{Mathematics Department\\ Princeton University\\ Fine Hall, Washington Road, Princeton, NJ 08544-1000, USA.}
\email{naor@math.princeton.edu}


\maketitle



This note contains a simple and elementary observation in response to~\cite{Bas17}. To explain it, we need to first briefly recall standard notation (introduced in~\cite{Mat90}) related to Lipschitz extension moduli.

Suppose that $(\MM,d_\MM)$ and $(\NN,d_\NN)$ are metric spaces and $\sub\subset \MM$. The Lipschitz constant of a mapping $\phi:\sub\to \NN$ is denoted $\|\phi\|_{\Lip(\sub;\NN)}$. Thus, $\|\phi\|_{\Lip(\sub;\NN)}\in [0,\infty]$ is the infimum over those $L\in [0,\infty]$ such that $d_\NN(\phi(x),\phi(y))\le Ld_\MM(x,y)$ for all $x,y\in \sub$.  Denote by $\ee(\MM,\sub;\NN)\in [1,\infty]$ the infimum over those $K\in [1,\infty]$ such that for every $\phi:\sub\to \NN$ with $\|\phi\|_{\Lip(\sub;\NN)}<\infty$ there exists $\Phi:\MM\to \NN$ that extends $\phi$, i.e., $\Phi(x)=\phi(x)$ for all $x\in \sub$,  and satisfies $\|\Phi\|_{\Lip(\MM;\NN)}\le K\|\phi\|_{\Lip(\sub;\NN)}$. Note that when $\NN$  is complete, $\NN$-valued Lipschitz functions on $\sub$ automatically extend to the closure of $\sub$ while preserving the Lipschitz constant, so one usually assumes here that $\sub$ is closed. Given $n\in \N$, the finitary modulus $\ee_n(\MM;\NN)$ is defined to be the supremum of $\ee(\MM,\sub;\NN)$ over all those subsets $\sub\subset \MM$ of cardinality at most $n$. Analogously, denote by $\ee^n(\MM;\NN)$ the supremum of $\ee(\sub\cup\{x_1,\ldots,x_n\},\sub;\NN)$ over all closed subsets $\sub\subset \MM$ and all $x_1,\ldots,x_n\in \MM\setminus \sub$.

The Lipschitz extension modulus $\ee_n(\MM;\NN)$ has been investigated extensively over the past several decades, though major fundamental questions about it remain open; a thorough description of what is known in this context appears in~\cite{NR17}. A variant of the modulus $\ee^n(\MM;\NN)$ (related to the stronger requirement that Lipschitz retractions exist) was studied in~\cite{Gru60,Lin64}. However, it seems that the quantity $\ee^n(\MM;\NN)$ did not receive further scrutiny in the literature, prior to the recent preprint~\cite{Bas17}. The purpose of this note is to derive the following simple upper bound on $\ee^n(\MM;\NN)$ in terms of $\ee_n(\MM;\NN)$, thus allowing one to use the available literature on $\ee_n(\MM;\NN)$ to bound $\ee^n(\MM;\NN)$, and in particular to improve some of the estimates in~\cite{Bas17}; see Remark~\ref{rem:quote} below. Many natural questions related to upper and lower bounds on  $\ee^n(\MM;\NN)$  remain open and warrant future investigation.

\begin{claim} \label{claim} $\ee^n(\MM;\NN)\le \ee_n(\MM;\NN)+2$ for every $n\in \N$ and every two metric spaces $(\MM,d_\MM),(\NN,d_\NN)$.
\end{claim}

\begin{remark}\label{rem:quote} Fix an integer $n\ge 3$. By combining Claim~\ref{claim} with~\cite[Theorem~1.10]{LN05}, it follows that\footnote{We use throughout the following (standard) asymptotic notation. Given two quantities $Q,Q'>0$, the notations
$Q\lesssim Q'$ and $Q'\gtrsim Q$ mean that $Q\le \mathsf{K}Q'$ for some
universal constant $\mathsf{K}>0$. The notation $Q\asymp Q'$
stands for $(Q\lesssim Q') \wedge  (Q'\lesssim Q)$. If  we need to allow for dependence on certain parameters, we indicate this by subscripts. For example, in the presence of an auxiliary parameter $\psi$, the notation $Q\lesssim_\psi Q'$ means that $Q\le c(\psi)Q' $, where $c(\psi) >0$ is allowed to depend only on $\psi$, and similarly for the notations $Q\gtrsim_\psi Q'$ and $Q\asymp_\psi Q'$.}
$$
\ee^n(\MM;Z)\lesssim \frac{\log n}{ \log\log n}
$$
 for every metric space $\MM$ and every Banach space $Z$. This answers (for sufficiently large $n$) a question that was asked in~\cite[page~3]{Bas17}. By combining Claim~\ref{claim} with~\cite[Theorem~2.12]{MP84} it follows that
 $$
 \forall\, p\in (1,2],\qquad \ee^n(\ell_p;\ell_2)\lesssim_p (\log n)^{\frac{1}{p}-\frac12}.
 $$
 Also, by combining Claim~\ref{claim} with~\cite{JL84} it follows that $\ee^n(\ell_2;Z)\lesssim\sqrt{\log n}$ for every Banach space $Z$. A ``dual'' version of this estimate follows by combining Claim~\ref{claim} with~\cite[Theorem~1.12]{LN05}, which yields that $\ee^n(Z;\ell_2)\lesssim\sqrt{\log n}$, thus improving (for sufficiently large $n$) over the bound $\ee^n(\ell_2;Z)\le\sqrt{n+1}$ of~\cite[Theorem~1.2]{Bas17}; more generally, this implies that $\ee^n(\ell_p;Z)\lesssim_p (\log n)^{1/p}$ for every $p\in (1,2]$.
 \end{remark}

\begin{proof}[Proof of Claim~\ref{claim}] Fix $\d\in (0,1)$, a closed subset $\sub\subset \MM$, and $x_1,\ldots,x_n\in \MM\setminus \sub$.  Since $\sub$ is closed, we have $d_\MM(x_j,\sub)>0$ for all $j\in \N$. Hence, there exist $y_1,\ldots,y_n\in \sub$ such that
\begin{equation}\label{eq:all j}
\forall\, j\in \n,\qquad d_\MM(y_j,x_j)\le (1+\d)d_\MM(x_j,\sub).
\end{equation}

Suppose that $\phi:\sub\to \NN$ is a Lipschitz mapping. Denote
\begin{equation*}\label{eqLdef KL}
K\eqdef\ee_n(\MM;\NN)\qquad\mathrm{and}\qquad L\eqdef\|\phi\|_{\Lip(\sub;\NN)}.
\end{equation*}
There is $\Psi:\{y_1,\ldots,y_n\}\cup\{x_1,\ldots,x_n\}\to \NN$ such that $\Psi(y_j)=\phi(y_j)$ for all $j\in \n$ and
\begin{equation}\label{eq:psi lip}
\|\Psi\|_{\Lip(\{y_1,\ldots,y_n\}\cup\{x_1,\ldots,x_n\};\NN)}\le (1+\d)K\|\phi\|_{\Lip(\{y_1,\ldots,y_n\};\NN)}\le (1+\d)KL.
\end{equation}
Define $\Phi:\sub\cup\{x_1,\ldots,x_n\}\to \NN$ by setting
$$
\Phi(z)\eqdef \left\{\begin{array}{ll}\Psi(z)&\mathrm{if}\ z\in \{x_1,\ldots,x_n\},\\
\phi(z)&\mathrm{if}\ z\in \sub.\end{array}\right.
$$
By design, $\Phi$ extends both $\phi$ and $\Psi$. For every $z\in \sub\setminus \{y_1,\ldots,y_n\}$ and $j\in \n$ we have
\begin{align}
\nonumber d_\NN\big(\Phi(z),\Phi(x_j)\big)&\le d_\NN\big(\Phi(z),\Phi(y_j)\big)+d_\NN\big(\Phi(y_j), \Phi(x_k)\big)\\
\label{use 23}
&\le Ld_\MM(z,y_j)+(1+\d)KL d_\MM(x_j,y_j)\\ \label{use nearest} &\le
L\big(d_\MM(z,x_j)+d_\MM(x_j,y_j)\big)+(1+\d)^2KLd_\MM(x_j,\sub)\\\label{use nearest 2}&\le
L\big(d_\MM(z,x_j)+(1+\d)d_\MM(x_j,\sub)\big)+(1+\d)^2KLd_\MM(x_j,z)\\  \nonumber &\le
L\big(d_\MM(z,x_j)+(1+\d)d_\MM(x_j,z)\big)+(1+\d)^2KLd_\MM(x_j,z)  \\ &=\big(2+\d+(1+\d)^2K\big)Ld_\MM(z,x_j),\label{that's all}
\end{align}
where~\eqref{use 23} uses the definition of $L$ and~\eqref{eq:psi lip}, and both~\eqref{use nearest} and~\eqref{use nearest 2}   use~\eqref{eq:all j}. Since $\Phi$ extends both $\phi$ and $\Psi$, it is $L$-Lipschitz on $\sub$ and $(1+\d)KL$-Lipschitz on $\{x_1,\ldots,x_n\}$. Therefore, due to~\eqref{that's all} we have $\|\Phi\|_{\Lip(\sub\cup\{x_1,\ldots,x_n\};\NN)}\le (2+\d+(1+\d)^2K)L$. Hence $\ee^n(\MM;\NN)\le 2+\d+(1+\d)^2\ee_n(\MM;\NN)$ and the desired estimate follows by letting $\d\to 0$.
\end{proof}

\begin{remark} By~\cite[Theorem~1.1]{Bas17} we have $\ee^n(\MM,\NN)\le n+1$ for every $n\in \N$ and all pairs of metric spaces $(\MM,d_\MM), (\NN,d_\NN)$. At the same time, it could be the case that $\ee_n(\MM,\NN)=\infty$; this is so for example when $n=2$, $\MM=\R$ and $\NN=\{0,1\}$, because there is no nonconstant continuous function from $\R$ to $\{0,1\}$. Hence, in general one cannot reverse the assertion of Claim~\ref{claim} so as to obtain an estimate of the form $\ee_n(\MM,\NN)\le f(\ee^{g(n)}(\MM,\NN))$ for some $f:[1,\infty)\to [1,\infty)$ and $g:\N\to \N$. However, it would be worthwhile to obtain good asymptotic bounds on such $f$ and $g$ for meaningful subclasses of the possible metric spaces $(\MM,d_\MM), (\NN,d_\NN)$, e.g.~when $\NN$ is a Banach space.
\end{remark}

\bibliographystyle{abbrv}
\bibliography{from-from-to-to}

 \end{document}